\documentclass[a4paper,12pt]{article}

\usepackage{amsmath, amsthm, amsfonts}
\usepackage[utf8]{inputenc}
\usepackage[colorlinks]{hyperref}

\newcommand{\abs}[1]{\left\vert#1\right\vert}

\newcommand{\norm}[1]{\left\Vert#1\right\Vert}

\def\RR{\mathbb{R}}
\def\N{\mathbb{N}}
\def\R{\mathbb{R}}
\def\pa{\partial}
\def\var{\varepsilon}

\newtheorem{thm}{Theorem}[section]

\newtheorem{lem}[thm]{Lemma}
\newtheorem{prp}[thm]{Proposition}
\newtheorem{hyp}[thm]{Hypothesis}
\theoremstyle{definition}
\newtheorem{dfn}[thm]{Definition}
\theoremstyle{remark}
\newtheorem{rem}[thm]{Remark}

\date{June 29, 2009}

\begin{document}
\begin{center}
\Large{Regularity and mass conservation for discrete
  coagulation-fragmentation equations with diffusion}
 \end{center}
\bigskip

\centerline{\scshape J. A. Ca\~{n}izo}
\medskip
{\footnotesize
 \centerline{Departament de Ma\-te\-m\`a\-ti\-ques}
 \centerline{Universitat Aut\`onoma de Barcelona, E-08193 Bellaterra, Spain}
\centerline{Email: \texttt{canizo@mat.uab.es}}}

\medskip
\centerline{\scshape L. Desvillettes }
\medskip
{\footnotesize
  \centerline{CMLA, ENS Cachan, IUF \& CNRS, PRES UniverSud}
  \centerline{61 Av. du Pdt. Wilson, 94235 Cachan Cedex, France}
\centerline{Email: \texttt{desville@cmla.ens-cachan.fr}}}
\medskip
\centerline{\scshape K. Fellner }
\medskip
{\footnotesize
  \centerline{DAMTP, CMS, University of Cambridge}
  \centerline{Wilberforce Road, Cambridge CB3 0WA, United Kingdom}
  \centerline{Email: K.Fellner@damtp.cam.ac.uk}
  \centerline{On leave from: Faculty of Mathematics, University of Vienna}
  \centerline{Nordbergstr. 15, 1090 Wien, Austria}
  \centerline{Email: \texttt{Klemens.Fellner@univie.ac.at}}
}
\bigskip
\centerline{\today}
\medskip

\begin{abstract}
  We present a new a-priori estimate for discrete
  coagulation-frag\-men\-tation systems with size-dependent diffusion
  within a bounded, regular domain confined by homogeneous Neumann
  boundary conditions. Following from a duality argument, this
  a-priori estimate provides a global $L^2$ bound on the mass density
  and was previously used, for instance, in the context of
  reaction-diffusion equations.

  In this paper we demonstrate two lines of applications for such an
  estimate: On the one hand, it enables to simplify parts of the known
  existence theory and allows to show existence of solutions for
  generalised models involving collision-induced, quadratic
  fragmentation terms for which the previous existence theory seems
  difficult to apply.  On the other hand and most prominently, it
  proves mass conservation (and thus the absence of gelation) for
  almost all the coagulation coefficients for which mass conservation
  is known to hold true in the space homogeneous case.
\end{abstract}

\noindent{Subject Class: 35B45, 35Q72, 82D60}
\medskip


\section{Introduction}
We consider the time evolution of a physical system where a set of
particles can aggregate into groups of two or more, called
\emph{clusters}, and where these clusters can diffuse in space with a
diffusion constant which depends on their size. If we represent space
by an open bounded set $\Omega \subseteq \RR^N$ with regular boundary,
the initial-boundary problem for the concentrations $c_i = c_i(t,x)
\geq 0$ of clusters with integer size $i \geq 1$ at position $x \in
\Omega$ and time $t \geq 0$ is given by the discrete
coagulation-fragmentation system of equations with spatial diffusion
and homogeneous Neumann boundary conditions~:
\begin{subequations}
  \label{eq:cfd}
  \begin{align}
    \label{eq:cfd-eq}
    \partial_t c_i - d_i \Delta_x c_i = Q_i + F_i &\quad \text{ for } x
    \in \Omega, t \geq 0, i \in \N^*,
    \\
    \label{eq:cfd-boundary}
    \nabla_{\!x} c_i \cdot n = 0 &\quad \text{ for } x \in \partial \Omega,
    t \geq 0,  i \in \N^*,
    \\
    \label{eq:cfd-initial}
    c_i(0,x) = c_i^0(x) &\quad \text{ for } x \in \Omega,  i \in \N^*,
  \end{align}
\end{subequations}
where $n = n(x)$ represents a unit normal vector at a point $x
\in \partial \Omega$, $d_i$ is the diffusion constant for clusters of
size $i$, and
\begin{equation}
 \begin{split}
  \label{eq:defQF-3}
  Q_i \equiv Q_i[c] := Q_i^+ - Q_i^-
  :=&\ \frac{1}{2} \sum_{j=1}^{i-1} a_{i-j,j}\, c_{i-j}\, c_j
  - \sum_{j=1}^\infty a_{i,j}\, c_i\, c_j, \\
  F_i \equiv F_i[c] := F_i^+ - F_i^- 
  :=&\ \sum_{j=1}^\infty B_{i+j}\, \beta_{i+j,i}\, c_{i+j} 
  - B_i\, c_i.
\end{split}
\end{equation}
%
The parameters $B_i$, $\beta_{i,j}$ and $a_{i,j}$, for integers $i,j
\geq 0$, represent the total rate $B_i$ of fragmentation of clusters
of size $i$, the average number $\beta_{i,j}$ of clusters of size $j$
produced due to fragmentation of a cluster of size $i$, and the
coagulation rate $a_{i,j}$ of clusters of size $i$ with clusters of
size $j$. We refer to these parameters as \emph{the coefficients} of
the system of equations. They represent rates, so they are always
nonnegative; single particles do not fragment further, and mass should
be conserved when a cluster fragments into smaller pieces, so one
always imposes
\begin{subequations}
  \label{eq:hyps0}
  \begin{align}
    \label{hyp:coefs-positive}
    a_{i,j}=a_{j,i} \geq 0, \qquad \beta_{i,j} \geq 0,
    &\qquad (i,j \in \N^*),
    \\
    \label{hyp:coefs-positive-2}
    B_1 = 0,
    \qquad\
    B_i \geq 0,
    &\qquad (i \in \N^*),
    \\
    \label{hyp:frag-conserves-mass}
    i= \sum_{j=1}^{i-1} j\,\beta_{i,j},
    &\qquad (i \in \N, i \geq 2).
  \end{align}
\end{subequations}
In fact, the last condition \eqref{hyp:frag-conserves-mass} implies
the conservation of the total mass $\int_{\Omega} \sum_{i=1}^\infty
i\,c_i\,dx$, which becomes obvious from the following formal
\emph{fundamental identity} or \emph{weak formulation} of the
coagulation and fragmentation operators: Consider a sequence of
nonnegative numbers $\{c_i\}$, and define $Q_i$, $F_i$ as in
eqs. \eqref{eq:defQF-3}, then, for any sequence of numbers
$\varphi_i$,
\begin{equation}
\begin{split}
  \label{eq:fundamental-identity}
  \sum_{i=1}^\infty \varphi_i\, Q_i
  &=
  \frac{1}{2} \sum_{i=1}^\infty \sum_{j=1}^\infty
  a_{i,j}\, c_i\, c_j\, (\varphi_{i+j} - \varphi_i - \varphi_j),
  \\
  \sum_{i=1}^\infty \varphi_i\, F_i
  &=
  - \sum_{i=2}^\infty B_i c_i
  \left(
    \varphi_i - \sum_{j=1}^{i-1} \beta_{i,j} \varphi_j
  \right).
\end{split}
\end{equation}
As a (still formal) consequence for solutions $\{c_i\}$ of
(\ref{eq:cfd}) -- (\ref{eq:defQF-3}), one can calculate the time
derivative of the integral of the moment $\sum \varphi_i c_i$ to
obtain
\begin{equation}
  \label{eq:moment-derivative}
  \frac{d}{dt} \int_\Omega \sum_{i=1}^\infty \varphi_i c_i
  =
  \int_\Omega \sum_{i=1}^\infty \varphi_i (Q_i + F_i),
\end{equation}
since the integral of the diffusion part vanishes due to the
homogeneous Neumann boundary condition. By choosing $\varphi_i := i$
above and thanks to \eqref{hyp:frag-conserves-mass}, we have
$\sum_{i=1}^\infty i\, Q_i = \sum_{i=1}^\infty i\, F_i = 0$, and the
total mass is formally conserved :
\begin{equation}
  \label{eq:mass-conservation}
  \norm{\rho (t, \cdot)}_{L^1} = \int_\Omega \sum_{i=1}^\infty i  c_i(t,x) \,dx=
   \int_\Omega \sum_{i=1}^\infty i c_i^0(x) \,dx
  = \norm{\rho^0}_{L^1}
  \quad (t \geq 0).
\end{equation}

Our main aim in this work is to provide some new bounds on the
regularity of weak solutions for system (\ref{eq:cfd}) --
\eqref{eq:defQF-3} by means of techniques developed in the context
of reaction-diffusion equations
\cite{citeulike:3798030,citeulike:3973601,PSch}, and to give three
applications to those bounds, the main one proving rigorously (for
almost all the coefficients where this is true in the homogeneous
case) mass conservation \eqref{eq:mass-conservation} and thus the
absence of gelation, a well-known phenomenon in
coagulation-fragmentation models \cite{EMP02,ELMP}, where the formal
conservation of mass is violated as clusters of infinite size are
formed.  \medskip

In this paper we will work with the global weak solutions constructed
in \cite{LM02} under the assumption
\begin{equation}
  \label{eq:LM-condition}
  \lim_{j \to +\infty} \frac{a_{i,j}}{j}
  =
  \lim_{j \to +\infty} \frac{B_{i+j}\, \beta_{i+j,i}}{i+j} = 0,
  \qquad (\mathrm{for\ fixed} \ i \geq 1),
\end{equation}
which were later extended in \cite{citeulike:3955115} to the case of $\Omega = \RR^N$. The notion of solution is the following, which we take from \cite{LM02}:
\begin{dfn}
  \label{defi}
  A global weak solution $c = \{c_i\}_{i \geq 1}$ to \eqref{eq:cfd} -- \eqref{eq:defQF-3} is
  a sequence of functions $c_i: [0,+\infty) \times \Omega \to
  [0,+\infty)$ such that for each $T > 0$,
  \begin{gather}
    c_i \in \mathcal{C}([0,T]; L^1(\Omega)),
    \quad i \geq 1,
    \\
    \sum_{j=1}^\infty a_{i,j} c_i c_j  \in L^1([0,T]\times\Omega),
    \\
    \sup_{t \geq 0} \int_\Omega \bigg[ \sum_{i=1}^\infty i c_i(t,x) \bigg]\, dx
    \leq
    \int_\Omega \bigg[  \sum_{i=1}^\infty i c^0_i(x)  \bigg]\, dx,
  \end{gather}
  and for each $i \geq 1$, $c_i$ is a mild solution to the $i$-th
  equation in \eqref{eq:cfd-eq}, that is,
  \begin{equation}
    \label{eq:ci-solution}
    c_i(t)
    =
    e^{d_i A_1 t} c_i^0 + \int_0^t e^{d_i A_1(t-s)} Q_i[c(s)] \,ds,
    \quad t \geq 0,
  \end{equation}
  where $Q_i[c]$ is defined by \eqref{eq:defQF-3}, $A_1$ denotes
  the closure in $L^1(\Omega)$ of the unbounded linear operator $A$ of
  $L^2(\Omega)$ defined by
  \begin{equation}
    \label{eq:A-domain}
    D(A) := \{w \in H^2(\Omega)
    \mid
    \nabla w \cdot n = 0 \text{ on  } \partial \Omega \},
    \qquad Aw = \Delta w,
  \end{equation}
  and $e^{d_i A_1 t}$ is the $C_0$-semigroup generated by $d_i A_1$ in
  $L^1(\Omega)$.
\end{dfn}
The existence result of \cite{LM02} reads:
\begin{thm}[Lauren\c cot-Mischler]
  \label{thm:LM-existence}
  Assume hypotheses \eqref{eq:hyps0} and
  \eqref{eq:LM-condition} on the coagulation and fragmentation
  coefficients. Assume also that
  \begin{equation*}
    d_i > 0
    \quad \text{for all}\ i \geq 1,    
  \end{equation*}
  and that the non-negative initial datum has finite mass:
  \begin{equation*}
    c_i^0 \geq 0 \ \text{on}\ \Omega  \quad \text{ and } \quad\int_\Omega \sum_{i=1}^\infty i\,c^0_i < +\infty.
  \end{equation*}
  Then, there exists a global weak solution to the initial-boundary problem \eqref{eq:cfd} -- \eqref{eq:defQF-3} in the sense of
Definition \ref{defi}.
\end{thm}

Under the extra assumptions on the diffusion constants and the initial data
 \begin{gather}
    \label{hyp:diffusion-bounded}
    0 < \inf_i \{d_i\} =: d, \qquad D := \sup_i \{d_i\} < +\infty,
    \\
    \label{hyp:initial-L2}
    \sum_{i=1}^\infty i c_i^0 \in L^2(\Omega),
  \end{gather}
we are in fact able to prove the following $L^2$ bound on the mass density $\rho(t,x) :=  \sum_{i=1}^\infty i\,c_i(t,x)$: 
Denoting by $\Omega_T$ the cylinder $[0,T] \times \Omega$, we have
the \begin{prp}  \label{lem:mass-L2}
Assume that \eqref{eq:hyps0}, \eqref{eq:LM-condition}, \eqref{hyp:diffusion-bounded} and \eqref{hyp:initial-L2} hold. Then,
for all $T > 0$ the mass $\rho$ of a weak solution to system \eqref{eq:cfd} -- \eqref{eq:defQF-3}
(given by Theorem \ref{thm:LM-existence}) lies in $L^2(\Omega_T)$
and the following estimate holds:
\begin{equation} \label{imp}
\|\rho\|_{L^2(\Omega_T)} \le \bigg( 1 + \frac{\sup_i \{d_i\}}{\inf_i \{d_i\}} \bigg)\, T\, \|\rho(0,\cdot)\|_{L^2(\Omega)} .
\end{equation}
\end{prp}

\begin{rem} \label{rem:1.4} Note that the assumption
  \eqref{eq:LM-condition} is only included in Proposition
  \ref{lem:mass-L2} in order to ensure the existence of a weak
  solution via Theorem \ref{thm:LM-existence}. Without assumption
  \eqref{eq:LM-condition}, the bound \eqref{imp} would still hold for
  smooth solutions of a truncated version of system \eqref{eq:cfd} --
  \eqref{eq:defQF-3} uniformly with respect to the truncation. See
  \cite{LM02} for the details of such a truncation.
\end{rem}
\medskip

In addition to Proposition \ref{lem:mass-L2}, we give a new proof of
an $L^1$ bound of the various coagulation and fragmentation terms:
\begin{prp}
  \label{lem:L1-terms}
  We still assume that \eqref{eq:hyps0}, \eqref{eq:LM-condition},
  \eqref{hyp:diffusion-bounded} and \eqref{hyp:initial-L2} hold.
  Then, for all $T > 0$ and $i \in \N^*$ all the terms $Q_i^+$, $Q_i^-$,
  $F_i^+$ and $F_i^-$ associated to a weak solution to system
  \eqref{eq:cfd}--\eqref{eq:defQF-3} (given by Theorem
  \ref{thm:LM-existence}) lie in $L^1(\Omega_T)$ with a
  bound which depends in an explicit way on the coagulation and
  fragmentation coefficients, the diffusion coefficients, and the
  initial data $c_i^0$.
\end{prp}

\begin{rem}
  The fact that the terms $Q_i^+$, $Q_i^-$, $F_i^+$ and $F_i^-$
  associated to a weak solution are in $L^1(\Omega_T)$ is included in
  the definition of weak solution; the main content of Proposition
  \ref{lem:L1-terms} is the explicit dependence of the bounds on the
  coefficients and initial data, which can be used to obtain uniform
  estimates for approximated solutions as we show for instance in
  section \ref{sec:existence}. For details on the explicit $L^1$
  bounds we refer to the proof of Proposition \ref{lem:L1-terms} in
  section \ref{nae}.
\end{rem}

\begin{rem}\label{rem:1.5}
  The $L^1$ bounds on $Q_i^+$, $Q_i^-$, $F_i^+$ and $F_i^-$ require the assumption
  \eqref{eq:LM-condition} only to ensure existence. They would hold at
  the formal level (that is, for smooth solutions of a truncated
  system) under the less stringent assumption
  \begin{equation}\label{bou} 
    K_i := \sup_{j\in\N} \frac{B_{i+j}\, \beta_{i+j,i}}{i+j} < +\infty
    \qquad (i\in \N^*).
  \end{equation}
  Note that the above $L^1$ bound also holds when assumptions
  \eqref{eq:hyps0}, \eqref{eq:LM-condition} are replaced by the
  assumptions of Theorem \ref{thm:LM-existence} in \cite{LM02}, but
  the proof is then much more difficult as it requires an induction on
  $i$ which can be removed under our extra assumptions.
\end{rem}
\medskip

In section \ref{sec:existence}, as a first application of the bounds
obtained in Propositions \ref{lem:mass-L2} and \ref{lem:L1-terms}, we
give a very simple proof of existence of weak solutions to
\eqref{eq:cfd}--\eqref{eq:defQF-3} in dimension $N=1$ (that is, the
result of Theorem \ref{thm:LM-existence} in dimension $1$) under the
additional assumptions \eqref{eq:hyps0} and \eqref{eq:LM-condition}.
\medskip

Our main application of the Propositions \ref{lem:mass-L2} and
\ref{lem:L1-terms} is however related to the problem of conservation
of mass \eqref{eq:mass-conservation}, which holds rigorously for
solutions to a truncated system (see e.g \cite{LM02}). Nevertheless,
it is an important issue in coagulation-fragmentation theory whether
\eqref{eq:mass-conservation} holds for weak solutions of system
\eqref{eq:cfd} -- \eqref{eq:defQF-3} itself, or if
\eqref{eq:mass-conservation} is replaced by an inequality stating that
mass in non-increasing in time.  If at some time $t$, the identity
\eqref{eq:mass-conservation} does not hold any more, we say that
gelation occurs, which means from a physical point of view that a
macroscopic object has been created.  \bigskip

Our main result in section \ref{sec:mass} basically shows that (under
the assumptions \eqref{eq:hyps0} and \eqref{eq:LM-condition}) gelation
does not occur when the coagulation coefficients $a_{i,j}$ are at most
linear and, moreover, slightly sublinear far off the diagonal
$i=j$. More precisely, we prove mass conservation under the following
condition on the coefficients $a_{i,j}$:
\begin{hyp}
  \label{hyp:aij-almost-linear}
  There is some bounded function $\theta: [0,+\infty) \to (0,+\infty)$
  such that $\theta(x) \to 0$ when $x \to +\infty$ and
  \begin{equation}
    \label{eq:aij-condition}
    a_{i,j} \leq (i+j)\, \theta(j/i)
    \quad
    \text{ for all } j \geq i.
  \end{equation}
  (Or equivalently, by symmetry,
  \begin{equation*}
    a_{i,j} \leq (i+j)\, \theta(\max\{j/i,i/j\})
    \quad
    \text{ for all } i,j \geq 1.)
  \end{equation*}
\end{hyp}
\begin{thm} 
  \label{thm:mass-conservation}
  Assume that \eqref{eq:hyps0}, \eqref{eq:LM-condition},
  \eqref{hyp:diffusion-bounded}, and \eqref{hyp:initial-L2} hold. Also,
  assume Hypothesis \ref{hyp:aij-almost-linear}. Then, the weak solution to the
  system \eqref{eq:cfd} given by Theorem \ref{thm:LM-existence} has a
  superlinear moment which is bounded on bounded time intervals; this
  is, there is some increasing function $C = C(T) > 0$, and some
  increasing sequence of positive numbers $\{\psi_i\}_{i \geq 1}$ with
  \begin{equation}
    \label{eq:phi-super}
    \lim_{i \to \infty} \psi_i \to +\infty
  \end{equation}
  such that for all $T > 0$,
  \begin{equation}
    \label{eq:superlinear}
    \int_\Omega \sum_{i=1}^\infty i\, \psi_i c_i
    \leq C(T)
    \quad
    \text{ for all }
    t \in [0,T].
  \end{equation}
  As a consequence, under these conditions all weak solutions given by
  Theorem \ref{thm:LM-existence} of \eqref{eq:cfd} conserve mass:
  \begin{equation}
    \label{eq:mc}
    \int_\Omega \rho_0(x) \,dx
    =
    \int_\Omega \rho(t,x) \,dx
    \quad
    \text{ for all } t \geq 0.
  \end{equation}
\end{thm}

\begin{rem}[Admissible coagulation coefficients]
  Let us comment on Hypothesis \ref{hyp:aij-almost-linear}. First note
  that \eqref{hyp:aij-almost-linear} includes coefficients of the form
  $$ 
  a_{i,j}
  \le \text{Cst} \, ( i^{\alpha}\, j^{\beta} +  i^{\beta}\, j^{\alpha}) 
  $$
  for any $\alpha, \beta >0$ such that $\alpha + \beta \le 1$ (take
  $\theta(x) =x^{-\var}$ for $\var>0$ small enough). It is also
  satisfied when
  $$ 
  a_{i,j}
  \le \text{Cst} \, \bigg( \frac{i}{\phi(i)} + \frac{j}{\phi(j)} \bigg), 
  $$
  where $x \mapsto \phi(x)$ is any positive strictly increasing function (for $x$
  big enough), which goes to infinity at infinity, and such that $x \mapsto
  \frac{x}{\phi(x)}$ is also increasing (take $\theta(\lambda) =
  \phi(\lambda)^{-1/2}$). All the examples $\phi = \log(1+ \cdot)$,
  $\phi = \log(1+\cdot)\circ\log(1+ \cdot)$, \dots , $\phi =
  \log(1+ \cdot)\circ\dots\circ\log(1+ \cdot)))$ satisfy this
  condition. Likewise, condition (\ref{eq:aij-condition}) also holds when
(for $i,j\ge 2$)
  \begin{equation}\label{sta}
    a_{ij}
    \le
    \text{Cst}\, \left(
      i  \frac{R(\log j)}{\log i} + j  \frac{R(\log i)}{\log j}
    \right)
  \end{equation}
  for some nondecreasing function $R$ such that $x \mapsto R(x)/x$ is nonincreasing and tends to $0$ when $x
  \to +\infty$.
 Note indeed that when (\ref{sta}) holds,
\begin{equation}\label{stasta}
 \frac{a_{ij}}{i+j} \le \frac1{1+ j/i} \, \frac{R[\log(j/i) + \log i]}{\log i} + \frac{j/i}{1 + j/i} \frac{R[\log i]}{\log(j/i) + \log i}.
\end{equation}
Then, condition (\ref{eq:aij-condition}) is obtained by distinguishing the cases $i \ge j/i$ and
$i \le j/i$ in both terms of the right hand side of  (\ref{stasta}).
\par
Assumption (\ref{sta}) can even be replaced by
 $$ a_{ij} \le Cst\, \bigg(i\,\frac{R(\log(\log j))}{\log(\log i)} + {j}\,\frac{R(\log (\log i))}{\log(\log j)} \bigg), $$
with the same requirements on $R$ as previously.
\par
Note however that the linear coefficient $a_{ij} = i+j$ (or the coefficient
 $a_{ij} = \frac{i}{\log i}\, \log j +  \frac{j}{\log j}\, \log i$)
does not satisfy hypothesis (\ref{hyp:aij-almost-linear}), though one would expect that Thm. \ref{thm:mass-conservation}
still holds for such coefficients.
\end{rem}
\bigskip

Before introducing a generalised coagulation-fragmentation model and
thus, a third application of the Propositions \ref{lem:mass-L2} and
\ref{lem:L1-terms}, let us briefly review previous results on
existence theory and mass conservation for the
coagulation-fragmentation system \eqref{eq:cfd}. With some further
restrictions on the coefficients as compared to \cite{LM02}, existence
of solutions by means of $L^\infty$ bounds on the $c_i$ has been
proven in \cite{MR1454671, citeulike:3946307, citeulike:3955138,
  citeulike:3458165, citeulike:3946301}. A different technique was
used in \cite{citeulike:3955119} to prove that equation \eqref{eq:cfd}
is well posed, locally in time, and globally in time when the space
dimension $N$ is one, always assuming that the coagulation and
fragmentation coefficients are bounded.

In a recent work \cite{citeulike:3460338}, Hammond and Rezakhanlou
considered equation \eqref{eq:cfd} without fragmentation, and gave
$L^\infty$ bounds on moments of the solution (and as a consequence,
$L^\infty$ bounds on the $c_i$). This implies uniqueness and mass
conservation for some coagulation coefficients that grow at
most linearly as well as an alternative proof of the existence of
$L^\infty$ solutions by a-priori bounds on the $c_i$; for instance, if $\Omega = \R^N$ and diffusion coefficients $d_i$ are nonincreasing and satisfying \eqref{hyp:diffusion-bounded} and if moreover 
\begin{equation*}
  \sum_{i=1}^\infty i \, c_i^0 \in L^\infty(\R^N),
  \qquad
  \sum_{i=1}^\infty i^2 \, c_i^0 \in L^1(\R^N),
  \qquad
  a_{i,j} \leq C\, (i+j)
\end{equation*}
for some $C > 0$ and all $i,j \geq 1$, then they show that mass is
conserved for all weak solutions of eq. \eqref{eq:cfd} without
fragmentation. See \cite[Theorems 1.3 and 1.4]{citeulike:3460338} and
\cite[Corollary 1.1]{citeulike:3460338} for more details.

In the spatially homogeneous case, mass conservation is known for
general data with finite mass and coagulation coefficients including the critical linear case $a_{i,j} \le
\text{Cst} (i+j)$ (see, for instance,
\cite{citeulike:2972710,citeulike:2972715}).
\bigskip

We finally give a third application of the Propositions
\ref{lem:mass-L2} and \ref{lem:L1-terms}. 
As mentioned already in the Remarks \ref{rem:1.4} and
\ref{rem:1.5}, Propositions \ref{lem:mass-L2} and \ref{lem:L1-terms} (despite true without restrictions on the
coagulation coefficients $a_{i,j}$ for smooth approximating solutions) 
do not really improve the theory of existence of weak solutions for the usual models of coagulation-fragmentation like \eqref{eq:cfd} as the full assumption
\eqref{eq:LM-condition} are needed in passing to the limit in the
approximating solutions. At best they help provide simpler proofs in particular cases, as done in section \ref{sec:existence}.  \medskip

On the other hand, Propositions \ref{lem:mass-L2} and
\ref{lem:L1-terms} are well suited for the existence theory of more
exotic models, for instance, when fragmentation occurs due to binary
collisions between clusters. Then, the break-up terms are quadratic,
being proportional to the concentration of the two clusters which
collide. This leads to coagulation-fragmentation models where all
terms in the right hand side are quadratic.  \medskip

More precisely, we consider that clusters of size $k$ and $l$ collide
with a rate $b_{k,l}\ge 0$, leading to fragmentation. As a
consequence, clusters of size $i< \max\{k,l\}$ are produced, in
average, at a rate $\beta_{i,k,l}\ge 0$ in such a way that the mass is
conserved (that is, $\sum_{i < \max\{k,l\}} i \, \beta_{i,k,l} = k+l$).
This leads to the following system (for $t\in \R_+$, $x\in \Omega$ a bounded
regular open subset of $\R^N$):
\begin{multline}
  \label{cf2}
  \pa_t c_i - d_i \,\Delta_x c_i
  =\ \frac12 \sum_{k+l=i} a_{k,l}\, c_k\,c_l - \sum_{k=1}^{\infty}
  a_{i,k}  \,c_i\, c_k
  \\
  +\frac{1}{2} \sum_{k,l=1}^{\infty} \sum_{i<\max\{k,l\}}
  b_{k,l}\, c_k\,c_l \, \beta_{i,k,l}  
  - \!\sum_{k=1}^{\infty} b_{i,k}\,c_i\,c_k  \qquad (i \in \N^*),
\end{multline}
together with the initial and boundary conditions
\eqref{eq:cfd-boundary}, \eqref{eq:cfd-initial}.  For this model, the
set of assumptions (\ref{eq:hyps0}) is replaced by
\begin{subequations}
  \label{eq:hyps02}
  \begin{align}
    \label{q1}
    &a_{i,j}=a_{j,i} \geq 0,
    &&\quad (i,j \in \N^*),
    \\
    \label{q2}
    &\beta_{i,k,l} = \beta_{i,l,k} \geq 0,
    &&\quad (i,k,l \in \N^*, i < \max\{k,l\}),
    \\
    \label{q2.5}
    &b_{i,k} = b_{k,i} \geq 0, \quad b_{1,1} = 0, 
    &&\quad (i,k \in \N^*, i < k),
    \\
    \label{q3}
    &\sum_{i < \max\{k,l\}} i\,\beta_{i,k,l}=k+l,
    &&\quad (k,l \in \N^*).
  \end{align}
\end{subequations}

Because of the quadratic character of the fragmentation terms, the
inductive method for the proof of existence devised by Lauren\c
cot-Mischler \cite{LM02} seems difficult to adapt in this
case. The method presented in our first application can however be
adapted, provided that the dimension is $N=1$ and that the following
assumptions are made on the coefficients: \bigskip

\begin{hyp}\label{hyp:quadratic}
  Assume \eqref{eq:hyps02}, and suppose that the diffusion
  coefficients are uniformly bounded above and below
  (eq. \eqref{hyp:diffusion-bounded}) and that the initial mass lies
  in $L^2(\Omega)$ (eq. \eqref{hyp:initial-L2}). In place of
  \eqref{eq:LM-condition} we assume further that
\begin{align}\label{nas1}
&\lim_{l \to \infty} \frac{a_{k,l}}l = 0, 
\qquad \lim_{l \to \infty} \frac{b_{k,l}}l = 0,
&&({\hbox{ for fixed }} k \in \N^*),\\
\label{nas2}
&\lim_{l \to \infty} \sup_{k} \left\{\frac{b_{k,l}}{kl} \, \beta_{i,k,l} \right\} = 0. 
&&( {\hbox{ for fixed }} i \in \N^*), 
\end{align}
\end{hyp}
We define a solution to \eqref{cf2} along the same lines as in
Definition \ref{defi}:
\begin{dfn}
  \label{dfn:cf2-solution}
  A global weak solution $c = \{c_i\}_{i \geq 1}$ to  \eqref{cf2}, the boundary condition
  \eqref{eq:cfd-boundary} and the initial data \eqref{eq:cfd-initial} is a sequence of
  functions $c_i: [0,+\infty) \times \Omega \to [0,+\infty)$ such that
  for each $T > 0$,
  \begin{equation}
    c_i \in \mathcal{C}([0,T]; L^1(\Omega)),
    \quad i \geq 1,
  \end{equation}
  the four terms on the r.h.s. of \eqref{cf2} are in
  $L^1([0,T]\times\Omega)$,
  \begin{equation}
    \sup_{t \geq 0} \int_\Omega
    \bigg[ \sum_{i=1}^\infty i c_i(t,x) \bigg]\, dx
    \leq
    \int_\Omega \bigg[  \sum_{i=1}^\infty i c^0_i(x)  \bigg]\, dx,
  \end{equation}
  and for each $i \geq 1$, $c_i$ is a mild solution to the $i$-th
  equation in \eqref{cf2}, that is,
  \begin{equation*}
    c_i(t)
    =
    e^{d_i A_1 t} c_i^0 + \int_0^t e^{d_i A_1(t-s)} Z_i[c(s)] \,ds,
    \quad t \geq 0,
  \end{equation*}
  where $Z_i[c]$ represents the right hand side of \eqref{cf2} and
  $A_1$, $e^{d_i A_1 t}$ are the same as in Definition \ref{defi}.
\end{dfn}

We are now able to prove the following theorem:
\begin{thm}\label{th}
  Under Hypothesis \ref{hyp:quadratic} on the coefficients and initial
  data of the equation, and in dimension $N=1$, there exists a global
  weak solution to eq. (\ref{cf2}) satisfying
  $$
  c_i \in C([0,T], L^1(\Omega)) \cap 
  L^{3-\var}(\Omega_T)\qquad (\text{for all } i\in\N^*, T>0, \var>0),
  $$ 
  for which the four terms appearing in the right hand side of
  (\ref{cf2}) lie in $L^1(\Omega_T)$.
\end{thm}

\begin{rem}
  The method of proof unfortunately does not seem to provide existence
  in dimensions $N \geq 2$. Dimension $N=2$ looks in fact critical as
  it doesn't allow a-priori a bootstrap in the heat equation with
  right hand side in $L^1$. A possible line of proof could follow
  \cite{GV} in the context of reaction-diffusion equations.  In higher
  dimensions $N\ge 3$, assuming additionally a detailed balance
  relation between coagulation and fragmentation, an entropy based
  duality method as in \cite{citeulike:3798030} could be used to
  define global weak $L^2$ solutions (see also
  \cite{citeulike:3973601}).
\end{rem}
\bigskip

Our paper is built in the following way: Section \ref{nae} is devoted
to the proof of Propositions \ref{lem:mass-L2} and \ref{lem:L1-terms}.
Then Sections \ref{sec:existence}, \ref{sec:mass}, and
\ref{sec:quadratic} are each devoted to one of the three applications.
In particular, Theorem \ref{thm:mass-conservation} is proven in Section 4 first in a
particular case (with a very short proof), and then in complete
generality. Theorem \ref{th} is proven in Section 5.  Finally, an
Appendix is devoted to the proof of a Lemma of duality due to
M. Pierre and D. Schmitt (cf. \cite{PSch}), which is the key to
Proposition \ref{lem:mass-L2}.

\section{A new a priori estimate}\label{nae}
The solutions given in \cite{LM02} are constructed by approximating
the system (\ref{eq:cfd})--(\ref{eq:defQF-3}) by a truncated system
(the procedure consists in setting the coagulation and fragmentation
coefficients to zero beyond a given finite size, and  smoothing the
initial data) for which very regular solutions exist.  Then, uniform
estimates for the solutions of this approximate system are
proven. Finally, it is shown that these solutions have a subsequence
which converges to a solution to the original system. In the proofs
below it must be understood that the bounds are obtained for the
truncated system (in a uniform way) and then transfered to the weak
solution by a passage to the limit: the fact that this transfer can be
done (in the case of the total mass) without replacing the equality by
an inequality is the heart of our second application.  \medskip

We begin with the

\begin{proof}[Proof of Proposition \ref{lem:mass-L2}]
Using the fact that
\begin{equation*}
 \partial_t \rho - \Delta ( M \rho ) = 0,\qquad
    \inf_{i\in\N^*} \{d_i\}\le M(t,x) := 
\frac{\sum_{i=1}^\infty d_i\, i\, c_i}{\sum_{i=1}^\infty i\, c_i}
\le \sup_{i\in\N^*} \{d_i\},
\end{equation*}
we can deduce thanks to a Lemma of duality   (\cite[Appendix]{citeulike:3798030}) that $\rho \in L^2(\Omega_T)$, and more precisely
that 
$$ 
\|\rho\|_{L^2(\Omega_T)} \le \bigg( 1 + \frac{\sup_i \{d_i\}}{\inf_i \{d_i\}} \bigg)\, T\, \|\rho(0,\cdot)\|_{L^2(\Omega)},
$$ 
for all $T>0$. For the sake of completeness, the Lemma is recalled with its proof
 in the Appendix (Lemma  \ref{lem:nl-diffusion-L2-estimate}).
\end{proof}
\medskip

We now turn to the
\begin{proof}[Proof of Proposition \ref{lem:L1-terms}]
For $F_i^-$, it is clear that
  \begin{equation*}
    F_i^- \leq B_i \, \rho
    \in L^2([0,T] \times \Omega)
    \subseteq L^1([0,T] \times \Omega),
  \end{equation*}
  thanks to Proposition \ref{lem:mass-L2}. For $F_i^+$ we use
  eq. \eqref{bou} to write
  \begin{equation}
    \label{eq:F+bound}
    F_i^+
    \leq
    \sum_{j=1}^\infty
    \left( \frac{B_{i+j}\, \beta_{i+j,i}}{i+j} \right)
    (i+j)\, c_{i+j}
    \leq
    K_i \sum_{j=1}^\infty (i+j)\, c_{i+j}
    \leq
    K_i \,\rho,
  \end{equation}
  which is again in $L^2([0,T] \times \Omega)$, and hence in
  $L^1([0,T] \times \Omega)$.

  For the coagulation terms, we have, since each $c_i$ is less than
  $\rho$,
  \begin{equation}
    \label{eq:Q+bound}
    Q_i^+
    \leq
    \frac{1}{4} \sum_{j=1}^{i-1} a_{i-j,j}
    \left(c_{i-j}^2 + c_j^2 \right)
    \leq
    \frac{1}{2} \rho^2 \left( \sum_{j=1}^{i-1} a_{i-j,j} \right), 
  \end{equation}
  which is in $L^1([0,T] \times \Omega)$ as $\rho^2$ is, and the sum
  only has a finite number of terms. Finally, for $Q_i^-$ we use the
  fact that $Q_i^+$ and $F_i^+$ are already known to be integrable: Thus, from
  eq. \eqref{eq:cfd} integrated over $[0,T] \times \Omega$,
  \begin{multline*}
    \int_\Omega c_i(T,x) \,dx
    + \int_0^T\!\! \int_\Omega Q_i^-(t,x) \,dx \,dt
    \\
    \leq
    \int_\Omega c_i^0(x) \,dx
    +
    \int_0^T\!\! \int_\Omega Q_i^+(t,x) \,dx \,dt
    +
    \int_0^T\!\! \int_\Omega F_i^+(t,x) \,dx \,dt.
  \end{multline*}
  This proves our result.
\end{proof}

\section{First application: a simplified proof of existence of solutions in dimension 1}
\label{sec:existence}

We begin this section with the following corollary of Proposition
\ref{lem:L1-terms}, in the particular case of dimension $N=1$.

\begin{lem}
  \label{lem:Linfty-bound}
  Assume that the dimension $N = 1$, and that
(\ref{eq:hyps0}), (\ref{hyp:diffusion-bounded}), (\ref{hyp:initial-L2})  and (\ref{bou}) hold.
Then, for all $T \geq 0$, $i \in \N^*$ the concentrations  $c_i \in
  L^\infty([0,T] \times \Omega)$ (where $c_i$ are smooth
solutions  of a truncated version of (\ref{eq:cfd}) -- (\ref{eq:defQF-3}), the $L^\infty$ norm
being independent of the truncation).
\end{lem}

\begin{proof}[Proof of Lemma \ref{lem:Linfty-bound}]
  We carry out a bootstrap regularity argument. Thanks to Proposition
  \ref{lem:L1-terms}, we know that (for all $i \in \N^*$)
  \begin{equation*}
    \left( \partial_t -
      d_i \Delta \right) c_i \in L^1([0,T] \times \Omega).
  \end{equation*}
  Using for example the results in \cite{citeulike:3797975}, this implies that for
  any $\delta > 0$,
  \begin{equation}
    \label{eq:ci-L3-}
    c_i \in L^{3-\delta}([0,T] \times \Omega)
    \qquad (i \in \N^*).
  \end{equation}
  Now, eq. \eqref{eq:ci-L3-} shows that $Q_i^+$ is actually more
  regular: from (the first inequality in) \eqref{eq:Q+bound},
  \begin{equation}
    \label{eq:Qi+L3/2-}
    Q_i^+ \in L^{\frac{3}{2} - \frac{\delta}{2}} ([0,T] \times \Omega)
    \qquad
    \text{ for all }
    \delta > 0, i \in \N^*.
  \end{equation}
  In addition, we already knew from eq. \eqref{eq:F+bound} that (for all $i \in \N^*$)
  \begin{equation}
    \label{eq:F+L2}
    F_i^+
    \in L^2([0,T] \times \Omega),
  \end{equation}
  [for which we do not need to assume that the space dimension $N$ is
  $1$]. Consequently, omitting the negative terms (for all $i \in \N^*$, $\delta>0$), we can find $h_i$ such that
  \begin{equation*}
    \left( \partial_t -
      d_i \Delta \right) c_i
    \leq h_i \in L^{\frac{3}{2} - \frac{\delta}{2}} ([0,T] \times \Omega).
  \end{equation*}
  As the $c_i$ are positive, this implies that
  \begin{equation*}
    c_i \in L^p([0,T] \times \Omega)
    \quad \text{ for all } p \in [1, +\infty[, i \in \N^*.
  \end{equation*}
  Again from \eqref{eq:Q+bound},
  \begin{equation*}
    Q_i^+ \in L^{p} ([0,T] \times \Omega)
    \quad \text{ for all } p \in [1, +\infty[, i \in \N^*.
  \end{equation*}
  From this and \eqref{eq:F+L2}, we can find $h_i$ such that
  \begin{equation*}
    \left( \partial_t -
      d_i \Delta \right) c_i
    \leq h_i \in L^2 ([0,T] \times \Omega),
  \end{equation*}
  which implies in turn that $c_i \in L^\infty([0,T] \times \Omega)$ (for all $i \in \N^*$).
\end{proof}
\medskip

We now have the possibility to give a short proof of Theorem
\ref{thm:LM-existence} in dimension $1$ (and under the extra
assumptions \eqref{hyp:diffusion-bounded},
(\ref{hyp:initial-L2})). Recall that a proof for any dimension can be
found in \cite{LM02}.  \medskip

\begin{proof}[Short proof of Theorem \ref{thm:LM-existence} 
in 1D under the assumptions \eqref{hyp:diffusion-bounded} and \eqref{hyp:initial-L2}]\ \\
Consider a sequence $c_i^M$ of (regular) solutions to
a truncated version of system \eqref{eq:cfd} -- (\ref{eq:defQF-3}).
Thanks to Proposition \ref{lem:Linfty-bound}, we know that for each $i\in \N^*$, $ \sup_M \norm{c_i^M}_{L^\infty(\Omega_T)} < + \infty$.
Then (for each  $i\in \N^*$) there
is a subsequence of the $(c_i^M)_{M\in\N}$ (which we still denote by $(c_i^M)_{M\in\N}$),
and a function $c_i \in L^\infty(\Omega_T)$, such that
\begin{equation}
  \label{eq:weak-star-conv}
  c_i^M \overset{*}{\rightharpoonup} c_i
  \quad \text{weak-$*$ in } L^\infty (\Omega_T).
\end{equation}
Using Proposition \ref{lem:L1-terms}, we also see that (for
any fixed $i\in \N^*$), the $L^1(\Omega_T)$ norms of $C_i^{+,M}$,
$C_i^{-,M}$, $F_i^{+,M}$, $F_i^{-,M}$ (the coagulation and
fragmentation terms associated to $\{c_i^M\}$) are bounded
independently of $M$. Using eq. \eqref{eq:cfd-eq} and the properties of
the heat equation, one sees that for each $i\in \N^*$, the sequence
$\{c_i^M\}$ lies in a strongly compact subset of $L^1(\Omega_T)$. Hence, by
renaming our subsequence again, we may assume that
\begin{equation}
  \label{eq:L1-conv}
  c_i^M \to c_i
  \quad \text{ in } L^1 (\Omega_T)  \text{ strong },
  \text{ for all } i \in \N^*.
\end{equation}
In order to prove that $\{c_i\}$ is indeed a solution to
eq. \eqref{eq:cfd} -- \eqref{eq:defQF-3}, let us prove that all terms $F_i^{+,M}$,
$F_i^{-,M}$, $C_i^{+,M}$, $C_i^{-,M}$ converge to the corresponding
expressions for $c_i$, which we denote by $F_i^{+}$,
$F_i^{-}$, $C_i^{+}$, $C_i^{-}$, as usual.

\begin{enumerate}
\item Positive fragmentation term: for each fixed $i$, the sum
  \begin{equation*}
    F_i^{+,M} = \sum_{j=1}^\infty B_{i+j}\, \beta_{i+j,i}\, c^M_{i+j}
  \end{equation*}
  converges to $F_i^+$ in $L^1(\Omega_T)$ because the tails of the sum
  converge to 0 uniformly in $M$ (this is due to hypothesis
  \eqref{eq:LM-condition}):
\begin{align*} 
\int_0^T\!\!\int_{\Omega} \bigg| \sum_j B_{i+j} \, \beta_{i+j,i} (c_{i+j}^M - c_{i+j}) \bigg|\, dx dt
\le& 2 \left( \sup_{j \ge J_0} \left| \frac{ B_{i+j} \, \beta_{i+j,i}}{i+j} \right| \right)\,\rho\\
&+  \sup_{j \le J_0} \| c_{i+j}^M - c_{i+j} \|_{L^1(\Omega_T)} . 
\end{align*}
\item The negative fragmentation term is just a multiple of $c_i^M$,
  so the convergence in $L^1(\Omega_T)$ is given by
  \eqref{eq:L1-conv}.
\item For each fixed $i$, the positive coagulation term is a finite
  sum of terms of the form $a_{i,j} c_i^M c_j^M$. Thanks to
  \eqref{eq:weak-star-conv} and \eqref{eq:L1-conv}, this converges to
  $a_{i,j} c_i c_j$ in $L^1(\Omega_T)$.
\item The negative coagulation term is
  \begin{equation*}
    Q_i^{-,M} = c_i^M \sum_{j=1}^\infty a_{i,j}\, c_j^M.
  \end{equation*}
 Since $c_i^M$ converges to $c_i$ weak-$*$ in $L^\infty(\Omega_T)$, it
  is enough to prove that $\sum_{j=1}^\infty a_{i,j}\, c_j^M$ converges to
  $\sum_{j=1}^\infty a_{i,j}\, c_j$ strongly in $L^1(\Omega_T)$. Observing that
$$ \int_0^T\!\!\int_{\Omega} \bigg| \sum_j a_{i,j} (c_{j}^M - c_{j}) \bigg|\, dx dt
\le 2 \left( \sup_{j \ge J_0} \left| \frac{ a_{i,j}}{j} \right| \right)\,\rho
+  \sup_{j \le J_0} \| c_{j}^M - c_{j} \|_{L^1(\Omega_T)} , $$
we see thanks to (\ref{eq:LM-condition}) and (\ref{eq:L1-conv}) that this 
convergence indeed holds.
\end{enumerate}
\end{proof}

\section{Second application: mass conservation}
\label{sec:mass}

We begin this section with a very short proof of Theorem \ref{thm:mass-conservation} in a
particular case in order to show how estimate \eqref{imp} works. More
precisely, we consider the pure coagulation case with $a_{i,j} =
\sqrt{i\,j}$ and $B_i=0$ (no fragmentation), and with initial data
satisfy additionally $\sum_{i=0}^{\infty} i\, \log i \, c_i(0,x) \, dx
< +\infty$ (which is sightly more stringent than only assuming finite
initial mass).  \medskip

Then, using the weak formulation \eqref{eq:fundamental-identity} with $\varphi_i=\log(i)$ (and remembering that $\log(1 + x) \le \text{Cst}\, \sqrt{x}$)
\begin{align}
  \frac{d}{dt} \int_{\Omega}\sum_{i=1}^{\infty} i\, \log i \, c_i \, dx 
  &= \int_{\Omega} \sum_{i=1}^{\infty} \sum_{j=1}^{\infty}
  \sqrt{ij} \, c_i\, c_j
  \left( i\, \log (1+\frac{j}{i}) + j\,\log (1+\frac{i}{j})\right) dx  \nonumber\\
  &\le 2 \int_{\Omega} \sum_{i=1}^{\infty} \sum_{j=1}^{\infty} i\,j\, c_i\, c_j\, dx
  \le 2  \int_{\Omega} \rho(t,x)^2\, dx .
  \label{eq:mc-simpleproof}
\end{align}
As a consequence, we have for all $T>0$
\begin{equation*}
  \int_{\Omega} \sum_{i=0}^{\infty} i\, \log i \, c_i(T,x) \, dx
  \le \int_{\Omega} \sum_{i=0}^{\infty} i\, \log i \, c_i(0,x) \, dx
  +  2  \int_0^T\!\!\int_{\Omega} \rho(t,x)^2\, dxdt, 
\end{equation*}
which ensures the propagation of the moment $\int \sum_{i=0}^{\infty}
i\, \log i \, c_i(\cdot,x) dx$, and therefore gives a rigorous proof of
conservation of the mass for weak solutions of the system: no gelation
occurs.

Our general result is obtained through a refinement of this argument
under hypothesis \eqref{hyp:aij-almost-linear}. Before giving the
proof of Theorem \ref{thm:mass-conservation} we need two technical lemmas, which will
substitute the intermediate step in \eqref{eq:mc-simpleproof}.

\begin{lem}
  \label{lem:psi-slowly-growing}
  Let $\{\mu_i\}_{i \geq 1}$ and $\{\nu_i\}_{i \geq 1}$ be sequences
  of positive numbers such that $\{\mu_i\}$ is bounded,
  \begin{equation*}
    \sum_{i=1}^\infty \mu_i = +\infty
    \quad \text{ and } \quad
    \lim_{i \to +\infty} \nu_i = +\infty.
  \end{equation*}
  Then we can find a sequence $\{\xi_i\}_{i \geq 1}$ of nonnegative
  numbers such that
  \begin{gather*}
    \sum_{i=1}^\infty \xi_i = +\infty,
    \\
    \xi_i \leq \mu_i
    \quad \text{ and } \quad
    \psi_i := \sum_{j=1}^i \xi_j \leq \nu_i
    \quad \text{ for all } i \geq 1.
  \end{gather*}
\end{lem}

\begin{proof}
  We may assume that $\nu_i$ is nondecreasing, for otherwise we can
  consider $\tilde{\nu}_i := \inf_{j \geq i} \{ \nu_j \}$ instead of
  $\nu_i$. Then, in order to find $\xi_i$ it is enough to define
  recursively $\xi_0 := 0$ and, for $i \geq 1$,
  \begin{equation*}
    \xi_i :=
    \begin{cases}
      \mu_i  & \text{ if } \mu_i + \sum_{j=0}^{i-1} \xi_j \leq \nu_i,
      \\
      0      & \text{ otherwise}.
    \end{cases}
  \end{equation*}
  By construction, $\xi_i \leq \mu_i$ for all $i \geq 1$, and also
  $\sum_{j=1}^i \xi_j \leq \nu_i$ for $i \geq 1$, as we are assuming
  $\{\nu_i\}$ nondecreasing.

  To see that $\{\xi_i\}$ cannot be summable, suppose otherwise that
  $\sum_{i=1}^\infty \xi_i = S < +\infty$. Take a bound $M > 0$ of
  $\{\mu_i\}$, and choose an integer $k$ such that $\nu_i \geq S + M$
  for all $i \geq k$. Then, by definition,
  \begin{equation*}
    \xi_i = \mu_i
    \quad \text{ for all } i \geq k,
  \end{equation*}
  which implies that $\{\xi_i\}$ is not summable, as $\{\mu_i\}$ is
  not, and gives a contradiction.
\end{proof}


\begin{lem}
  \label{lem:choose-psi-i}
  Assume \eqref{eq:aij-condition}. There is a nondecreasing sequence
  of positive numbers $\{\psi_i\}_{i \geq 1}$ such that $\psi_i \to
  +\infty$ when $i \to +\infty$, and
  \begin{equation}
    \label{eq:less-than-ij}
    a_{i,j} (\psi_{i+j} - \psi_i)
    \leq
    C j
    \quad
    \text{ for all }
    i,j \geq 1,
  \end{equation}
  for some constant $C > 0$.

  In addition, for a given sequence of positive numbers $\lambda_i$ with $\lim_{i \to +\infty} \lambda_i = +\infty$, we can choose
  $\psi_i$ so that $\psi_i \leq \lambda_i$ for all $i$.
\end{lem}

\begin{proof}
  First, we may assume that the function $\theta$ given in Hypothesis
  \ref{hyp:aij-almost-linear} is nonincreasing on $[1,+\infty)$, as we can always
  take $\tilde{\theta}(x) := \sup_{y \geq x} \theta(y)$ instead.
  
  We choose a sequence of nonnegative numbers $\{ \xi_i \}$ by applying
  Lemma \ref{lem:psi-slowly-growing} with
  \begin{gather}
    \label{eq:mu_i}
    \mu_i := \frac{1}{(1+i) \log(1+i)},
    \\
    \label{eq:nu_i}
    \nu_i := \min\left\{
      \lambda_i,\,
      \frac{1}{\theta(\sqrt{i/2})},
    \right\}.
  \end{gather}
  Note that the conditions in Lemma \ref{lem:psi-slowly-growing} are
  met: the sequence in the right hand side of \eqref{eq:mu_i} is not
  summable, and the right hand side of \eqref{eq:nu_i} goes to
  $+\infty$ with $i$. If we define $\psi_i := \sum_{j=1}^i \xi_j$,
  then the following is given by Lemma \ref{lem:psi-slowly-growing}:
  \begin{gather*}
    \xi_i
    \leq
    \frac{1}{(1+i) \log(1+i)},
    \qquad
    \psi_i
    \leq
    \frac{1}{\theta(\sqrt{i/2})},
    \quad
    \psi_i \leq \lambda_i,
    \qquad
    i \geq 1,
    \\
    \lim_{i \to +\infty} \psi_i = +\infty.
  \end{gather*}
  These conditions essentially say that $\psi_i$ grows slowlier than
  $\log \log (i)$, slowlier than $\theta(\sqrt{i/2})^{-1}$, and slowlier
  than $\lambda_i$, yet still diverges as $i \to +\infty$.  \medskip

  We can now prove \eqref{eq:less-than-ij} to hold for these
  $\{\psi_i\}$ by distinguishing three cases:

  \noindent{1.} For any $i, j \geq 1$, as $\log(1+k) \geq 1/2$ for all
  $k \geq 1$,
  \begin{equation*}
    \psi_{i+j} - \psi_i = \sum_{k=i+1}^{i+j} \xi_{k}
    \leq
    2 \sum_{k=i+1}^{i+j} \frac{1}{1+k}      
    \leq 2 \log(i+j+1) - 2 \log(i+1)
    \leq \frac{2j}{i}.
  \end{equation*}
  Then, in case $j\le i$ we use the fact that
  $\theta(x) \leq C_\theta$ for some constant $C_\theta > 0$ and all $x > 0$ and have
  \begin{equation*}
    a_{i,j} (\psi_{i+j} - \psi_i)
    \leq
    2\, C_\theta (i+j) \frac{j}{i}
    \leq
    4\,C_\theta\, j,\qquad \text{for } j\le i.
  \end{equation*}
  \noindent{2.} Secondly, for $i < j \leq i^2$,
  \begin{align*}
    \psi_{i+j} - \psi_i &\leq \sum_{k=i+1}^{2i^2} \xi_{k} \leq
    \sum_{k=i+1}^{2i^2} \frac{1}{(k+1)\log(k+1)}
    \\
    &\leq \log \log (2i^2 + 1) - \log \log(i+1) \leq \log \bigg(\frac{2
      \log (\sqrt{3}i)}{\log(i+1)} \bigg)\leq C_1,
  \end{align*}
  for some number $C_1 > 0$. Thus,
  \begin{equation*}
    a_{i,j} (\psi_{i+j} - \psi_i)
    \leq
    C_1 C_\theta (i+j) \leq 2 C_1 C_\theta j.
  \end{equation*}
  \noindent{3.} Finally, for $j > i^2$,
  \begin{equation*}
    \psi_{i+j} - \psi_i
    \leq
    \psi_{i+j}
    = \sum_{k=1}^{i+j} \xi_k
    \leq
    \frac{1}{\theta(\sqrt{(i+j)/2})}
    \leq
    \frac{1}{\theta(\sqrt{j})},
  \end{equation*}
  and as $\theta$ is nonincreasing on $[1,+\infty)$ (we may assume
  this; see the beginning of this proof), we have for all $j > i^2$
  \begin{equation*}
    a_{i,j} (\psi_{i+j} - \psi_i)
    \leq
    (i+j) \theta(j/i) \frac{1}{\theta(\sqrt{j})}
    \leq
    (i+j) \theta(\sqrt{j}) \frac{1}{\theta(\sqrt{j})}
    = i+j \leq 2j.
  \end{equation*}
  Together, these three cases show \eqref{eq:less-than-ij} for all
  $i,j \geq 1$.
\end{proof}

Now we are ready to finish the proof of our result on mass
conservation:

\begin{proof}[Proof of Theorem \ref{thm:mass-conservation}]
  As remarked above (cf. beginning of section \ref{nae}), we will
  prove the estimate \eqref{eq:superlinear} for a regular solution to
  an approximating system, with a constant $C(T)$ that does not depend
  on the regularisation. Then, passing to the limit, the result is
  true for a weak solution thus constructed.

We consider a solution to an approximating system on $[0,+\infty)$, which we still denote by $\{c_i\}_{i \geq 1}$. 
Then, by a version of the de la Vallée-Poussin's Lemma, (see, for instance, Proposition 9.1.1 in \cite{C06} or also proof of Lemma 7 in \cite{citeulike:2964344}), there exists a
nondecreasing sequence of positive numbers $\{\lambda_i\}_{i \geq 1}$ (independent of the regularisation of the initial data) which diverges as $i \to +\infty$, and such that
  \begin{equation}
    \label{eq:phi-initially-finite}
    \int_\Omega \sum_{i=1}^\infty i \,\lambda_i c_i^0\,dx < +\infty.
  \end{equation}
If we define $r_i := \int_\Omega i c_i^0$,
note that this is just the claim that one can find $\lambda_i$ as above with $\sum_i \lambda_i r_i < +\infty$.
\medskip

Taking $\{\psi_i\}$ as given by Lemma \ref{lem:choose-psi-i}, such that $\psi_i \leq \lambda_i$ for all $i \geq 1$, 
we have thus $\int_\Omega \sum_{i=1}^\infty i \,\psi_i\, c_i^0(x) \,dx
< +\infty$. Then, as integrating over $\Omega$ makes the diffusion
term vanish due to the no-flux boundary conditions, we estimate
\begin{equation}    
\label{eq:mc1}
\frac{d}{dt} \int_\Omega \sum_{i=1}^\infty i \, \psi_i c_i\,dx
\leq\frac{1}{2} \int_\Omega \sum_{i,j=1}^\infty a_{i,j} c_i c_j((i+j) \psi_{i+j} - i\, \psi_i - j\,\psi_j)\,dx,
\end{equation}
where we used that the contribution of the fragmentation term is nonpositive, as can be seen from \eqref{eq:fundamental-identity}
with $\varphi_i \equiv i\,\psi_i$, and the fact that
\begin{equation*}
    \sum_{j=1}^{i-1} \beta_{i,j}\, j \psi_j
    \leq
    \psi_i \sum_{j=1}^{i-1} \beta_{i,j}\, j
    =
    i\,\psi_i,
  \end{equation*}
as $\psi_i$ is nondecreasing and \eqref{hyp:frag-conserves-mass} holds.
Continuing from \eqref{eq:mc1}, by the symmetry of the $a_{i,j}$, and using the inequality \eqref{eq:less-than-ij} from Lemma \ref{lem:choose-psi-i}, we have
\begin{equation}
\frac{d}{dt} \int_\Omega \sum_{i=1}^\infty i\,\psi_i\, c_i\,dx
\leq \int_\Omega \sum_{i,j=1}^\infty a_{i,j}\, c_i\, c_j
i\, (\psi_{i+j} - \psi_i)\,dx
\le C \int_\Omega \rho^2\,dx.
  \label{eq:mc2}
\end{equation}
Thus, Proposition \ref{lem:mass-L2} showing $\rho \in L^2(\Omega_T)$
proves that $\int_\Omega \sum_{i=1}^\infty i \, \psi_i\, c_i\,dx$ is
bounded on bounded time intervals. Mass conservation is a direct
consequence of this.
\end{proof}
\begin{rem}[Absence of gelation via tightness]
  It is interesting to sketch an alternative proof showing
  conservation of mass via a tightness argument and without
  establishing superlinear moments. By introducing the superlinear
  test sequence $i\phi_k(i)$ with $ \phi_k(i) = \frac{\log i}{\log k}
  1_{i<k} + 1_{i\ge k}$ for all $k \in \N^*$, we use the weak
  formulation \eqref{eq:fundamental-identity} to see (as above) that
  the fragmentation part is nonnegative for superlinear test
  sequences, and use the symmetry of the $a_{i,j}$ to reduce summation
  over the indices $i\ge j\in \N^*$, which leads
  to the estimate
  \begin{multline*}
    \frac{d}{dt} \int_\Omega \sum_{i=1}^{\infty} c_i\, i \phi_k(i)\,dx 
    \leq\int_\Omega \sum_{i\ge j}^\infty \sum_{j=1}^\infty
    {a_{i,j}}[i c_i] [c_j]
    \left(\frac{\log(1+\frac{j}{i})}{\log(k)}\,\mathbb{I}_{i<k}
    \right. \\
    \left.
      +\frac{j}{i}\left(\frac{\log(1+\frac{i}{j})}{\log(k)}\,\mathbb{I}_{i+j<k}+\frac{\log(\frac{k}{j})}{\log(k)}\,\mathbb{I}_{j<k\le i+j}\right)\right)dx.
  \end{multline*}
  For the first term, we use $\log(1+{j}/{i})\le {j}/{i}$. Then,
  for the second and third terms, we distinguish further the areas where $i/j\le
  \log(k)$ and $i/j> \log(k)$. When $i/j\le \log(k)$, we estimate
  $1+{i}/{j}=1+i/j\le 1+\log(k)$ and ${k}/{j}\le 1+{i}/{j}\le
  1+\log(k)$, respectively.  On the other hand, when $i/j > \log(k)$,
  both the second and the third term are bounded by one.
  Altogether, we get thanks to assumption
  \eqref{hyp:aij-almost-linear}, i.e. $\frac{a_{i,j}}{i}\le
  \text{Cst}\,{\theta(i/j)}$ for $i\le j$:
  \begin{align*}
    \frac{d}{dt} \int_\Omega \sum_{i=1}^{\infty} c_i\, i \phi_k(i)\,dx &\leq \left(\frac{1}{\log(k)}+\frac{\log(1+\log{k})}{\log(k)}\right) \sup\limits_{i\ge j\in\N^*} 
    \left\{\frac{a_{i,j}}{i}\right\}\int_\Omega \rho^2\,dx \\
    &\quad +\int_\Omega \sum_{i\ge j}^\infty \sum_{j=1}^\infty
    [i c_i] [j c_j] \frac{a_{i,j}}{i}\,\mathbb{I}_{i/j> \log(k);j<k}\,dx\\
    &\leq \text{Cst} \left(\frac{\log(1+\log{k})}{\log(k)}
      +\sup\limits_{i/j\ge \log(k)} {\theta\left(\frac{i}{j}\right)}\right)\int_\Omega \rho^2\,dx
  \end{align*}
  and the right hand side tends to zero as $k\to\infty$.  Hence, using
  Proposition~\ref{lem:mass-L2} and integrating over a time interval
  $[0,T]$, we get thanks to a tightness argument
 that the mass is indeed
  conserved, and no gelation occurs.
\end{rem}

\section{Third Application: Fragmentation due to collisions in
  dimension 1}
\label{sec:quadratic}

\begin{proof}[Proof of Theorem \ref{th}]
We introduce $(c_i^M)_M$  a sequence of smooth solutions for a truncated version of eq. (\ref{cf2}). We first observe that
Proposition \ref{lem:mass-L2} still holds thanks to the duality estimate, that is $\rho := \sum_i i\, c_i \in L^2(\Omega_T)$ for all $T>0$. Estimate (\ref{eq:Q+bound}), in which only the coagulation kernel appears, also holds. Moreover, thanks to (\ref{q3}),
$$ 
\sum_{k,l}\sum_{\\max\{k,l\}>i} b_{k,l}\, c_k\,c_l \, \beta_{ikl}
\le \text{Cst}_i \sum_k\sum_{l} (k+l)\, c_k\,c_l  \le \text{Cst}_i\, \rho^2 \in L^1(\Omega_T). 
$$
The loss terms
$$ 
\sum_{k=1}^{\infty} a_{i,k}  \,c_i\, c_k,\qquad
\sum_{k=1}^{\infty} b_{i,k}\,c_i\,c_k 
$$ 
lie then in $L^1(\Omega_T)$ by integration of the
equation on $[0,T] \times \Omega$.
\bigskip

Using now eq. (\ref{cf2}), we see that  (for all $i\in \N^*$)
$ \pa_t c_i^M - d_i \pa_{xx} c_i^M  $ belongs to a bounded subset of $L^1(\Omega_T)$. As a
consequence, $c_i^M$ belongs (for all $i\in \N^*$) to a compact subset of $L^{3-\var}([0,T] \times \Omega)$ for all $T>0$
and $\var>0$. We denote (for all $i\in \N^*$) by $c_i$ a limit (in $L^{3-\var}([0,T] \times \Omega)$ strong)
 of a subsequence of $(c_i^M)_{M\in\N}$ (still denoted by  $(c_i^M)_{M\in\N}$).
\bigskip

We now pass to the limit in all terms of the r.h.s. of eq. (\ref{cf2}). The first term can easily be dealt with, since it consists of a finite sum. Then, we pass to the limit in the second term:
\begin{gather*}
\int_0^T\!\! \int_{\Omega} \bigg| \sum_{k=1}^{\infty} a_{i,k} \, c_i^n  \, c_k^n 
 - \sum_{k=1}^{\infty} a_{i,k} \, c_i  \, c_k \bigg| \, dx dt \\
\le \int_0^T\!\! \int_{\Omega} \bigg| \sum_{k=1}^{K} a_{i,k} \, c_i^n  \, c_k^n
 - \sum_{k=1}^{K} a_{i,k} \, c_i  \, c_k \bigg| \, dx dt 
 +\, 2 \, \|\rho\|_{L^2}^2 \,  \sup_{k > K} 
\left\{\frac{a_{i,k}}k\right\} . 
\end{gather*}

The second part of this expression is small when $K$ is large enough thanks to assumption (\ref{nas1}), (\ref{nas2}),
 while the first part tends to $0$ 
for all given $K$. 
\bigskip

The fourth term of the r.h.s. of eq. (\ref{cf2}) can be treated exactly in the same way.
We now turn to the third term:
\begin{gather*}
 \int_0^T\!\! \int_{\Omega} \bigg| 
\sum_{k,l=1}^{\infty}\sum_{i<\max\{k,l\}}  b_{k,l}\, c_k^n\,c_l^n \, \beta_{i,k,l}
- \sum_{k,l=1}^{\infty}\sum_{i<\max\{k,l\}}  b_{k,l}\, c_k\,c_l \, \beta_{i,k,l} \bigg| \, dx dt \\
 \le \int_0^T\!\! \int_{\Omega} \bigg| 
\sum_{k,l=1}^{K}\sum_{i<\max\{k,l\}}^{k\le K, l \le K}  b_{k,l}\, c_k^n\,c_l^n \, \beta_{i,k,l}
- \sum_{k,l=1}^{K}\sum_{i<\max\{k,l\}}^{k\le K, l \le K}  b_{k,l}\, c_k\,c_l \, \beta_{i,k,l} \bigg| \, dx dt \\
 +\,  4\, \|\rho\|_{L^2}^2 \, \sup_{l\ge K} \sup_{k \in \N} \left\{\frac{b_{k,l}}{kl} \, \beta_{i,k,l}\right\} . 
\end{gather*}

Once again, the second term is small when $K$ is large enough thanks
to assumption (\ref{nas1}), (\ref{nas2}), while the first term tends
to $0$ for all given $K$.
\end{proof}

\section*{Acknowledgements}

KFs work has been supported by the KAUST Award No. KUK-I1-007-43, made
by King Abdullah University of Science and Technology (KAUST). JAC was
supported by the project MTM2008-06349-C03-03 of the Spanish
\emph{Ministerio de Ciencia e Innovaci\'{o}n}.
 LD was supported by the
french project ANR CBDif.
The authors acknowledge partial support of the trilateral project
Austria-France-Spain (Austria: FR 05/2007 and ES 04/2007, Spain:
HU2006-0025 and HF2006-0198, France: Picasso 13702TG and Amadeus 13785 UA).
 LD and KF also wish to acknowledge
the kind hospitality of the CRM of Barcelona.

\section{Appendix: A duality lemma}
We recall here results from e.g. \cite{PSch, citeulike:3798030}.
We start with the

\begin{lem}
  \label{lem:dual-bound}
  Assume that $z: \Omega_T \to [0, +\infty)$ satisfies
  \begin{alignat}{2}
    \partial_t z + M \Delta z &= - H
    &\qquad& \text{ on } \Omega,
\nonumber    \\
    \nabla z \cdot n &= 0
    && \text{ on } \partial \Omega,
\label{eq:nl-diffusion-dual0}
    \\
    z(T,x) &= 0
    && \text{ on } \Omega,
\nonumber
  \end{alignat}
where $H \in L^2(\Omega_T)$, and $d_1 \geq M \geq d_0 > 0$. Then,
\begin{equation}
    \label{eq:dual-estimate}
    \norm{ z(0,\cdot) }_{L^2(\Omega)}
    \leq
    \left( 1 + \frac{d_1}{d_0} \right)
    T \norm{H}_{L^2(\Omega_T)}.
  \end{equation}
\end{lem}
\begin{proof}[Proof of Lemma \ref{lem:dual-bound}]
  Calculating the time derivative of $\int_\Omega \abs{\nabla z}^2$,
  or alternatively multiplying eq. \eqref{eq:nl-diffusion-dual0} by
  $\Delta z$ and integrating on $\Omega$, we obtain
  \begin{equation*}
    -\frac{1}{2} \frac{d}{dt} \int_\Omega \abs{\nabla z}^2\,dx
    + \int_\Omega M (\Delta z)^2\,dx
    =
    \int_\Omega - H \Delta z\,dx,
  \end{equation*}
  where the boundary condition on $z$ was used. Integrating on
  $[0,T]$ and taking into account that $z(T,x) = 0$,
  \begin{align}
    \frac{1}{2} \int_\Omega \abs{\nabla z(0,\cdot)}^2\,dx
    + \int_{\Omega_T} M (\Delta z)^2\,dxdt 
    &=
    \int_{\Omega_T} H \Delta z\,dxdt \nonumber\\
    &\leq
    \norm{H}_{L^2(\Omega_T)} \norm{\Delta z}_{L^2(\Omega_T)}.
    \label{eq:p2}
  \end{align}
Using that $M \geq d_0$ we see that $\int_{\Omega_T} M (\Delta z)^2
  \geq d_0 \norm{\Delta z}_{L^2(\Omega_T)}^2$, so \eqref{eq:p2} implies
  \begin{equation*}
    d_0 \norm{ \Delta z }_{L^2(\Omega_T)}
    \leq
    \norm{H}_{L^2(\Omega_T)}.
  \end{equation*}
  From this and \eqref{eq:nl-diffusion-dual0} we have
  \begin{align*}
    \norm{ \partial_t z }_{L^2(\Omega_T)}
    &\leq
    \norm{ M \Delta z }_{L^2(\Omega_T)} + \norm{H}_{L^2(\Omega_T)}
    \\
    &\leq
    d_1 \norm{ \Delta z }_{L^2(\Omega_T)} + \norm{H}_{L^2(\Omega_T)}
    \leq
    \left( 1 + \frac{d_1}{d_0} \right) \norm{H}_{L^2(\Omega_T)}.
  \end{align*}
  Finally,
\begin{equation*}
    \norm{z(0,\cdot)}_{L^2(\Omega)}
    \leq
    \int_0^T \norm{\partial_s z_s}_{L^2(\Omega)} \,ds
    \leq
    \left( 1 + \frac{d_1}{d_0} \right)\, T \norm{H}_{L^2(\Omega_T)}.
  \end{equation*}
\end{proof}

\begin{lem}
  \label{lem:nl-diffusion-L2-estimate}
  Assume that $\rho: \Omega_T \to [0, +\infty)$ and satisfies
  \begin{alignat}{2}
    \label{eq:nl-diffusion}
    \partial_t \rho - \Delta (M \rho) &\leq 0
    \qquad \text{ on } \Omega,
    \\
    \nabla (\rho\,M) \cdot n &= 0
    \qquad \text{ on } \partial \Omega,\nonumber
  \end{alignat}
  where $M: \Omega_T \to \RR$ is a function which satisfies $d_1 \geq
  M \geq d_0 > 0$ for some numbers $d_1$, $d_0$.
  Then,
  \begin{equation*}
    \norm{\rho}_{L^2(\Omega_T)}
    \leq
    \left( 1 + \frac{d_1}{d_0} \right)\, T \norm{\rho(0,\cdot)}_2.
  \end{equation*}
\end{lem}
\begin{proof}[Proof of Lemma \ref{lem:nl-diffusion-L2-estimate}]
  Consider the dual problem  \eqref{eq:nl-diffusion-dual0}
    -- \eqref{eq:dual-estimate}
  for an arbitrary function $H \in L^2(\Omega_T)$, with $H \geq
  0$. Then, $z\ge 0$, 
  and integrating by parts in eq. \eqref{eq:nl-diffusion-dual0}, one finds
  that
  \begin{align*}
    \int_{\Omega_T} \rho H\,dxdt
    &=
    - \int_{\Omega_T} \rho (\partial_t z + M \Delta z)\,dxdt
    \\
    &=
    \int_{\Omega_T} z (\partial_t \rho - \Delta(\rho M))\,dxdt
    + \int_\Omega \rho(0,\cdot) \,z(0,\cdot)\,dxdt
    \leq
    \int_\Omega  \rho(0,\cdot) \,z(0,\cdot)\,dxdt,
  \end{align*}
  where we have used eq. \eqref{eq:nl-diffusion}, eq.  \eqref{eq:dual-estimate} and the boundary
  conditions on $\rho\, M$ and $z$.
   Hence, for any nonnegative function $H \in L^2(\Omega_T)$,
  \begin{equation*}
    \int_{\Omega_T} \rho H\,dxdt
    \leq
    \norm{\rho(0,\cdot)}_{L^2(\Omega)} \norm{z(0,\cdot)}_{L^2(\Omega)},
  \end{equation*}
  and thanks to  Lemma \ref{lem:dual-bound},
  \begin{equation*}
    \int_{\Omega_T} \rho H\,dxdt
    \leq
    ( 1 + d_1/d_0)\, T \norm{\rho(0,\cdot)}_{L^2(\Omega)} \norm{H}_{L^2(\Omega_T)}.
  \end{equation*}
Remembering that $\rho \ge 0$, we obtain
  by duality:
  \begin{equation*}
    \norm{\rho}_{L^2(\Omega_T)}
    \leq
     (1 + d_1/d_0)\, T  \norm{\rho(0,\cdot)}_{L^2(\Omega)} .
  \end{equation*}
  This proves the lemma.
\end{proof}

\bibliographystyle{abbrv}
\bibliography{bibliography3}

\begin{thebibliography}{10}

\bibitem{citeulike:3955119}
H.~Amann.
\newblock Coagulation-fragmentation processes.
\newblock {\em Archive for Rational Mechanics and Analysis}, 151(4):339--366,
  2000.

\bibitem{citeulike:2972710}
J.~M. Ball and J.~Carr.
\newblock The discrete coagulation-fragmentation equations: Existence,
  uniqueness, and density conservation.
\newblock {\em Journal of Statistical Physics}, 61(1):203--234, 1990.

\bibitem{MR1454671}
P.~B\'{e}nilan and D.~Wrzosek.
\newblock On an infinite system of reaction-diffusion equations.
\newblock {\em Advances in Mathematical Sciences and Applications},
  7(1):351--366, 1997.

\bibitem{C06}
J.~A. Ca\~{n}izo.
\newblock {\em Some problems related to the study of interaction kernels:
  coagulation, fragmentation and diffusion in kinetic and quantum equations}.
\newblock PhD thesis, Universidad de Granada, June 2006.

\bibitem{citeulike:2972715}
J.~Carr and F.~Costa.
\newblock Asymptotic behavior of solutions to the coagulation-fragmentation
  equations. {I}{I}. {W}eak fragmentation.
\newblock {\em Journal of Statistical Physics}, 77(1):89--123, 1994.

\bibitem{citeulike:2964344}
J.~Carrillo, L.~Desvillettes, and K.~Fellner.
\newblock Exponential decay towards equilibrium for the inhomogeneous
  {A}izenman-{B}ak model.
\newblock {\em Communications in Mathematical Physics}, 278(2):433--451, 2008.

\bibitem{citeulike:3946307}
J.~F. Collet and F.~Poupaud.
\newblock Existence of solutions to coagulation-fragmentation systems with
  diffusion.
\newblock {\em Transport Theory and Statistical Physics}, 25(3):503--513, 1996.

\bibitem{citeulike:3797975}
L.~Desvillettes and K.~Fellner.
\newblock Entropy methods for reaction-diffusion equations: slowly growing
  a-priori bounds.
\newblock {\em Revista Matem\'{a}tica Ibero\-americana}, 24(2):407--431, 2008.

\bibitem{citeulike:3798030}
L.~Desvillettes, K.~Fellner, M.~Pierre, and J.~Vovelle.
\newblock About global existence for quadratic systems of reaction-diffusion.
\newblock {\em Journal of Advanced Nonlinear Studies}, 7(3):491--511, 2007.

\bibitem{ELMP}
M.~Escobedo, P.~Lauren\c{c}ot, S.~Mischler, and B.~Perthame.
\newblock Gelation and mass conservation in coagulation-fragmentation models.
\newblock {\em Journal of Differential Equations}, 195(1):143--174, 2003.

\bibitem{EMP02}
M.~Escobedo, S.~Mischler, and B.~Perthame.
\newblock Gelation in coagulation and fragmentation models.
\newblock {\em Communications in Mathematical Physics}, 231:157--188, 2002.

\bibitem{GV}
T.~Goudon and A.~Vasseur.
\newblock Regularity analysis for systems of reaction-diffusion equations.
\newblock {\em to appear in Annales de l'\'{E}cole Normale Sup\'{e}rieure}.

\bibitem{citeulike:3955138}
F.~Guia\c{s}.
\newblock Convergence properties of a stochastic model for
  coagulation-fragmentation processes with diffusion.
\newblock {\em Stochastic Analysis and Applications}, 19(2):245--278, 2001.

\bibitem{citeulike:3460338}
A.~Hammond and F.~Rezakhanlou.
\newblock Moment bounds for the smoluchowski equation and their consequences.
\newblock {\em Communications in Mathematical Physics}, 276(3):645--670, 2007.

\bibitem{LM02}
P.~Lauren\c{c}ot and S.~Mischler.
\newblock Global existence for the discrete diffusive coagulation-fragmentation
  equation in ${L}^1$.
\newblock {\em Revista Matemática Iberoamericana}, 18:731--745, 2002.

\bibitem{citeulike:3973601}
M.~Pierre.
\newblock Weak solutions and supersolutions in for reaction-diffusion systems.
\newblock {\em Journal of Evolution Equations}, 3(1):153--168, 2003.

\bibitem{PSch}
M.~Pierre and D.~Schmitt.
\newblock Blowup in reaction-diffusion systems with dissipation of mass.
\newblock {\em SIAM Review}, 42:93--106, 2000.

\bibitem{citeulike:3458165}
D.~Wrzosek.
\newblock Existence of solutions for the discrete coagulation-fragmentation
  model with diffusion.
\newblock {\em Topological Methods in Nonlinear Analysis}, 9(2):279--296, 1997.

\bibitem{citeulike:3946301}
D.~Wrzosek.
\newblock Mass-conserving solutions to the discrete
  coagulation–frag\-mentation model with diffusion.
\newblock {\em Nonlinear Analysis}, 49(3):297--314, 2002.

\bibitem{citeulike:3955115}
D.~Wrzosek.
\newblock Weak solutions to the cauchy problem for the diffusive discrete
  coagulation–fragmentation system.
\newblock {\em Journal of Mathematical Analysis and Applications},
  289(2):405--418, 2004.

\end{thebibliography}

\end{document}